\documentclass[11pt,twoside]{amsart}
\usepackage[latin1]{inputenc}
\usepackage[OT1]{fontenc}
\usepackage{amsmath}
\usepackage{amsthm}
\usepackage{amssymb}
\usepackage[all]{xy}

\newtheorem{thm}{Theorem}[section]

\newtheorem{prop}[thm]{Proposition}

\newtheorem{lem}[thm]{Lemma}
\newtheorem{cor}[thm]{Corollary}

\numberwithin{equation}{section}

\theoremstyle{definition}

\newtheorem{remark}[thm]{Remark}

\newcommand{\qqed}{\hspace*{\fill}$\Box$}

\newcommand{\Db}{{\rm D}^{\rm b}}

\newcommand{\Pic}{{\rm Pic}}

\newcommand{\rk}{{\rm rk}}

\newcommand{\End}{{\rm End}}
\newcommand{\Hom}{{\rm Hom}}

\newcommand{\id}{{\rm id}}

\newcommand{\Ext}{{\rm Ext}}

\newcommand{\Hilb}{{\rm Hilb}}

\newcommand{\kl}{{\mathcal L}}

\newcommand{\ko}{{\mathcal O}}

\newcommand{\IC}{\mathbb{C}}

\newcommand{\IP}{\mathbb{P}}
\newcommand{\IQ}{\mathbb{Q}}

\newcommand{\IZ}{\mathbb{Z}}

\renewcommand{\to}{\xymatrix@1@=15pt{\ar[r]&}}
\renewcommand{\rightarrow}{\xymatrix@1@=15pt{\ar[r]&}}
\renewcommand{\mapsto}{\xymatrix@1@=15pt{\ar@{|->}[r]&}}
\renewcommand{\twoheadrightarrow}{\xymatrix@1@=15pt{\ar@{->>}[r]&}}
\renewcommand{\hookrightarrow}{\xymatrix@1@=15pt{\ar@{^(->}[r]&}}
\newcommand{\congpf}{\xymatrix@1@=15pt{\ar[r]^-\sim&}}
\renewcommand{\cong}{\simeq}

\begin{document}

\title[Fourier--Mukai partners of canonical covers...]{Fourier--Mukai partners of canonical covers of bielliptic and Enriques surfaces}
\author[P.\ Sosna]{Pawel Sosna}

\address{Mathematisches Institut,
Universit{\"a}t Bonn, Endenicher Allee 60, 53115 Bonn, Germany}
\email{sosna@math.uni-bonn.de}

\subjclass[2000]{18E30, 11G10, 14J28}

\keywords{Derived categories, Fourier--Mukai partners, bielliptic surfaces, Enriques surfaces}

\begin{abstract} \noindent
We prove that the canonical cover of an Enriques surface does not admit non-trivial Fourier--Mukai partners. We also show that the canonical cover of a bielliptic surface has at most one non-isomorphic Fourier--Mukai partner. The first result is then applied to birational Hilbert schemes of points and the second to birational generalised Kummer varieties. An appendix establishes that there are no exceptional or spherical objects in the derived category of a bielliptic surface.

\vspace{-2mm}\end{abstract}
\maketitle

\maketitle

\section{Introduction}

In \cite{Mukai} Mukai discovered that an abelian variety $A$ and its dual abelian variety $\widehat{A}$ are always derived equivalent (or Fourier--Mukai partners), even though in general they are not even birational. Since this observation a lot of effort has been put into the investigation of possible Fourier--Mukai partners of a given variety $X$. It turns out that, for example, the derived category determines the variety if the canonical (or anti-canonical) bundle is ample. Thus, it is natural to consider the case where the canonical bundle is trivial or torsion. Since the derived category of an elliptic curve determines the curve, the case of surfaces is the next interesting one. If the Kodaira dimension of the surface is 0, then, as is well known, there are four possibilities. The surface is either abelian, K3, Enriques or bielliptic. Bridgeland and Maciocia proved in \cite{BM} that an Enriques or a bielliptic surface does not admit non-isomorphic Fourier--Mukai partners. In the other two cases quite a lot is known, see, for example, \cite{HLOY}, \cite{Orlov} and \cite{HLOY2}. In this note we prove the following result (see Propositions \ref{K3-Enriques} and \ref{abelian-bielliptic}).

\begin{thm} If $X$ is a K3 surface which covers an Enriques surface, then any surface derived equivalent to $X$ is isomorphic to $X$. If $A$ is an abelian surface covering a bielliptic surface, then the only non-trivial Fourier--Mukai partner $A$ can have is $\widehat{A}$.
\end{thm}

Thus the additional geometric information imposes restrictions on the number of Fourier--Mukai partners. We will use the above result to establish that birational Hilbert schemes of points on K3 surfaces as above are automatically isomorphic, see Corollary \ref{bir-Hilb}. Corollary \ref{bir-kum} establishes a similar result in the case of birational generalised Kummer varieties.

The note is organised as follows: In Section 2 we recall the necessary facts about lattices and canonical covers and describe the known formulae for the number of Fourier--Mukai partners of a K3 resp.\ an abelian surface. In Section 3 we prove the first part of the theorem and Section 4 establishes the second part. The last section is devoted to the study of the derived category of a bielliptic surface: We prove that neither spherical nor exceptional objects exist. Throughout we work over the complex numbers.

\smallskip

\noindent{\bf Acknowledgements.} I thank Daniel Huybrechts for his comments on a preliminary version of this paper and Mart\'{i} Lahoz, Emanuele Macr\`{i} and Sven Meinhardt for useful discussions. This work was partially financially supported by the SFB/TR 45 `Periods, Moduli Spaces and Arithmetic of Algebraic Varieties' of the DFG (German Research Foundation).

\section{Preliminaries}\label{Preliminaries}
A \emph{lattice} is a free abelian group $L$ of finite rank endowed with a symmetric non-degenerate $\IZ$-valued bilinear form $b$. A lattice is \emph{even} if for any $l\in L$ the integer $b(l,l)$ is even. An \emph{isometry} between two lattices is a group homomorphism preserving the bilinear forms. The \emph{dual lattice} $L^*$ is the group $\Hom(L,\IZ)$ endowed with the natural extension of the bilinear form on $L$. There is an embedding $L\, \hookrightarrow L^*$ given by $l \mapsto b(-,l)$ and $L$ is called \emph{unimodular} if the map is an isomorphism. An example of a unimodular lattice is the \emph{hyperbolic plane} $U$, which is the group $\IZ^2$ endowed with the bilinear form which on the basis $e$ and $f$ is given by $e^2=f^2=0$ and $ef=1$. Another example is the unique positive definite even unimodular lattice $E_8$, see \cite[Ch.\ I.2]{BPV}. If $L$ is a lattice and $k \in \IZ$, then $L(k)$ denotes the same abelian group with the bilinear form multiplied by $k$. Given a sublattice $M$ of a lattice $L$, its orthogonal complement $M^\bot$ is the group of elements $l \in L$ satisfying $b(m,l)=0$ for all $m\in M$. We call a sublattice $M$ of a lattice $L$ \emph{primitive} if $L/M$ is torsion-free. 

The \emph{discriminant group} of the lattice $L$ is by definition $A_L=L^*/L$. This is a finite group of order $|\det(L)|$ and $b$ induces a symmetric bilinear form $b_L\colon A_L\times A_L\rightarrow \IQ/\IZ$ and a corresponding quadratic form $q_L\colon A_L\rightarrow \IQ/2\IZ$. A lattice $L$ is \emph{$p$-elementary} if $A_L\cong (\IZ/p\IZ)^a$ for some natural number $a$. 

Given a lattice $L$, its \emph{genus} is the set $\mathcal{G}(L)$ of isometry classes of lattices $L'$ such that $(A_L,q_L)\cong (A_{L'},q_{L'})$ and the signature of $L'$ is equal to the signature of $L$.

The following result of Nikulin, see \cite[Thm.\ 14.4.2]{Nikulin}, will be used quite frequently in the sequel.

\begin{prop}[Nikulin]\label{Nik}
Let $T$ be an even indefinite nondegenerate lattice satisfying the following conditions:
\begin{itemize}
\item[(a)] $\rk(T)\geq \rk(A_{T_p})+2$ for all prime numbers $p\neq 2$.
\item[(b)] if $\rk(T)=l(A_{T_2})$, then $q_{T_2}$ contains a component $u(2)$ or $v(2)$.
\end{itemize}
Then the genus of $T$ contains only one class, and the homomorphism $O(T)\rightarrow O(A_T)$ is surjective. Here $A_{T_p}$ denotes the $p$-component of the finite abelian group $A_T$, $l$ denotes the number of generators, $u(2)$ is the discriminant group of the lattice $U(2)$ and $v(2)$ is described in \cite{Nikulin}.\qqed
\end{prop}

We will frequently need the formulae for the number of Fourier--Mukai partners of a K3 surface or an abelian surface established in \cite{HLOY} and \cite{HLOY2} respectively. Let $X$ be a K3 or abelian surface with transcendental lattice $T_X$ (sometimes written as $T(X)$) and period $\IC\omega_X$. We define the group
\[G_{Hodge}:=O_{Hodge}(T_X,\IC\omega_X)=\left\{g \in O(T_X)\; | \; g(\IC\omega_X)=\IC\omega_X\right\}.\]
It is known that the genus of a lattice with fixed rank and discriminant is a finite set. We have a natural map $O(L)\rightarrow O(A_L)$. On the other hand, given a marking $\varphi$ for $X$, we can use it to define an embedding $G_{Hodge}\hookrightarrow O(T)$, where $T=\varphi(T_X)$. Using that the discriminant groups of a lattice and its orthogonal complement are isomorphic, we get an action of $G_{Hodge}$ on $O(A_{T^\bot})$. 

Consider a K3 surface $X$ and set $\mathcal{G}(NS(X))=\mathcal{G}(L)=\left\{L_1,\ldots,L_k\right\}$. Then the number of Fourier--Mukai partners $FM(X)$ of $X$ is given by
\begin{equation}\label{K3-number}
|FM(X)|=\sum_{i=1}^k |O(L_i)\setminus O(A_{L_i})/G_{Hodge}|.
\end{equation}
If $A$ is an abelian surface we have a surjective morphism
\[\xi\colon FM(A)\rightarrow \mathcal{P}^{eq}(T(A),U^{\oplus 3}), B \mapsto \iota_B,\]
so that $\xi(\iota_B)=\left\{B,\widehat{B}\right\}$ and where the set on the right is the set of $G_{Hodge}$-equivalence classes ($G_{Hodge}$ is defined as above) of primitive embeddings of $T(A)$ into $U^{\oplus 3}$. We furthermore have
\begin{equation}\label{abelian-number}
|\mathcal{P}^{eq}(T(A),U^{\oplus 3})|=\sum_{i=1}^k |O(L_i)\setminus O(A_{L_i})/G_{Hodge}|,
\end{equation}
where the $L_i$ are the lattices in the genus of $NS(A)$.

We also need to recall the notion of \emph{canonical cover}. Namely, let $X$ be a smooth projective variety with torsion canonical bundle $\omega_X$ whose order is $n$. The canonical cover $\widetilde{X}$ is the unique (up to isomorphism) smooth projective variety with trivial canonical bundle with an \'etale map $\pi\colon \widetilde{X} \rightarrow X$ of degree $n$ such that $\pi_*\ko_{\widetilde{X}}=\bigoplus_{i=0}^{n-1}\omega_X^i$. Furthermore, there is a free action of the cyclic group $G=\IZ/n\IZ$ on $\widetilde{X}$ such that the morphism $\pi$ is the quotient morphism. 

\section{K3 surfaces covering Enriques surfaces}

Recall that an Enriques surface is a compact complex surface $S$ of Kodaira dimension 0 with $H^1(S,\ko_S)=H^2(S,\ko_S)=0$. Any Enriques surface is projective, its canonical bundle is torsion of order 2 and the canonical cover of an Enriques surface is a K3 surface. Conversely, a quotient of a K3 surface by a fixed-point free involution is an Enriques surface. Note that we have an isomorphism $\Pic(S)\cong H^2(S,\IZ)$ obtained from the exponential sequence. Dividing out torsion we get the lattice $E_8(-1)\oplus U$. Pullback to the covering K3 surface gives the lattice $E:=E_8(-2)\oplus U(2)$, which is often referred to as the \emph{Enriques lattice}. An Enriques surface is \emph{generic} if the Picard group of the covering K3 surface is precisely $E$. 

By general results  in \cite{Mukai2} (which are based on Proposition 2.1) we know that a K3 surface of Picard rank $\geq 12$ does not have any Fourier--Mukai partners. Thus, if we consider K3 surfaces covering Enriques surfaces we only have to consider Picard ranks 10 and 11. 

\begin{prop}\label{K3-Enriques}
Let $S$ be an Enriques surface and let $X$ be the covering K3 surface. Then $\Db(X)\cong \Db(Y) \Leftrightarrow X\cong Y$.
\end{prop}  

\begin{proof}
By the Derived Torelli theorem for K3 surfaces the existence of a derived equivalence between $X$ and $Y$ is equivalent to the existence of a Hodge isometry between the transcendental lattices $T(X)\cong T(Y)$. Let us begin with generic Enriques surfaces. It is a fact that the transcendental lattice of $E$ in the K3 lattice $\Lambda=E_8(-1)^{\oplus 2}\oplus U^{\oplus 3}$ is isometric to $E\oplus U$. Now, \cite[Thm.\ 1.4]{Namikawa} gives that any isometry of $E\oplus U$ extends to an isometry of $\Lambda$. By the Torelli theorem for K3 surfaces we conclude that $X\cong Y$.

Ohashi \cite[Prop.\ 3.5]{Ohashi} classified the Neron--Severi lattices of K3 surfaces of Picard rank 11 covering an Enriques surface. We have the following two series
\begin{align*}
&F_N:=U(2)\oplus E_8(-2)\oplus \left\langle -2N\right\rangle, \;\; N\geq 2,\\
&G_M:=U\oplus E_8(-2)\oplus \left\langle -4M\right\rangle, \;\; M\geq 1.
\end{align*}
In the second case, the Neron--Severi group contains the hyperbolic plane and hence any isometry of the transcendental lattice extends to the K3 lattice by \cite[Thm.\ 14.4.4]{Nikulin}. Therefore, the K3 surfaces belonging to the second case do not have any FM-partners. Consider the lattices $F_N$. It is clear that for any $p\neq 2$ the rank of the $p$-component of the discriminant group is $1$. Hence, condition (a) in Proposition \ref{Nik} is satisfied. One also easily sees that condition (b) is satisfied as well. Equation \ref{K3-number} gives the result. 
\end{proof}

\begin{cor}\label{bir-Hilb}
Let $X$ and $Y$ be two K3 surfaces covering Enriques surfaces and assume that the Hilbert schemes of $n$ points $\Hilb^n(X)$ and $\Hilb^n(Y)$ are birational. Then there exists an isomorphism $\Hilb^n(X)\cong\Hilb^n(Y)$.
\end{cor}

\begin{proof}
The assumption implies that $\Db(X)\cong \Db(Y)$, see \cite[Prop.\ 10]{Ploog}. Now apply the above proposition. Of course, the same argument works, for example, for Hilbert schemes of points on K3 surfaces of Picard rank at least 12.
\end{proof}

We can apply this result to \emph{Enriques manifolds} which were introduced in \cite{OS}. An Enriques manifold is a connected complex manifold which is not simply connected and whose universal cover is a hyperk\"ahler manifold. A particular example is obtained as follows. Let $S$ be an Enriques surface and let $X$ be the K3 surface covering it. For any odd $n\geq 1$ the induced action of the group $G=\IZ/2\IZ$ (corresponding to the involution on $X$) on $\Hilb^n(X)$ is free and the quotient is an Enriques manifold $R$ with $\pi_1(R)=\IZ/2\IZ$.

\begin{cor}
Let $R=\Hilb^n(X)/G$ and $R'=\Hilb^n(X')/G$ be birational Enriques manifolds, where $X$ and $X'$ cover generic Enriques surfaces. Then there exists an isomorphism $R\cong R'$.
\end{cor}  

\begin{proof}
The universal covers are birational as well and the claim follows at once from the previous corollary and the observation that under our assumption the surface $X$ ($\cong X'$) admits only one fixed-point free involution.
\end{proof}

The situation changes if one also considers twisted FM-partners of a given K3 surface $X$. Recall that given a K3 surface $X'$ and a class $\alpha$ in the Brauer group of $X'$, one can consider the abelian category of $\alpha$-twisted sheaves on $X'$ and its bounded derived category $\Db(X',\alpha)$. A \emph{twisted} FM-partner of $X$ is a twisted K3 surface $(X',\alpha)$ such that there is an equivalence $\Db(X',\alpha)\cong \Db(X)$ (see, for example, \cite{HS} for more information). We have an explicit formula in \cite{Ma} which allows us to compute the number of twisted FM-partners $\text{FM}^d(X)$ of $X$ for any given order $d$ of $\alpha$. If the Neron--Severi lattice is 2-elementary, as is the case for the Enriques lattice, then \cite[Cor.\ 4.5]{Ma} states that there are no twisted FM-partners for $d\neq 1,2$. Applying this to a K3 surface $X$ with $\Pic(X)=E$, we have
\[\text{FM}^2(X)= |O_{Hodge}(T(X)) \setminus I^2(A_{T(X)})|,\]
where $I^2(A_{T(X)})$ is the set of elements in $A_{T(X)}$ of order $2$. We know that $O_{Hodge}(T(X))$ is a cyclic group whose Euler function value divides 12, the rank of $T(X)$. Thus, $|O_{Hodge}(T(X))|\leq 42$. On the other hand, it is easily checked that $I^2(A_{T(X)})$ has more than 42 elements. Thus, the cover $X$ of a generic Enriques surface has a twisted Fourier--Mukai partner. It is to be expected that the same holds for higher Picard numbers. 

If one twists $X$ as well, then for any natural number $N$ there exist $N$ non-isomorphic algebraic K3 surfaces $X_1,\ldots, X_N$ of Picard rank 20, which can be assumed to be Kummer surfaces, and Brauer classes $\alpha_1,\ldots,\alpha_N$ on these surfaces such that the twisted derived categories $\Db(X_i,\alpha_i)$ are all derived equivalent, see \cite[Prop.\ 8.1]{HS}. Since any Kummer surface covers an Enriques surface by \cite{Keum}, we see that allowing twisting creates arbitrarily many twisted FM-partners. 

\section{Abelian surfaces covering bielliptic surfaces}

Recall that a bielliptic surface is a complex projective surface $S$ of Kodaira dimension 0 with $H^1(S,\ko_S)=\IC$ and $H^2(S,\ko_S)=0$. It turns out that any such surface is a quotient of a product of two elliptic curves $E\times F$ by the action of a finite group $G$. Note that we allow $E$ and $F$ to be isomorphic. The group $G$ acts on $E$ by translations and on $F$ in such a manner that $F/G\cong \IP^1$, so, in particular, it does not act by translations only. The canonical bundle of $S$ is torsion of order $2,3,4$ or $6$. In fact, there are only the following possibilities for $G$ and some of them include restrictions on $F$ (see \cite[Ch.\ V.5]{BPV}).

\begin{itemize}
\item[(1)] The group $G$ is cyclic of order $n=2,3,4$ or $6$. The order of $G$ is equal to the order of the canonical bundle of $S$ and the canonical cover $\widetilde{S}$ is isomorphic to $E\times F$, where $F$ has complex multiplication for $n=3,4$ and $6$. 
\item[(2)] The group $G$ is $\IZ/3\IZ \times \IZ/3\IZ$. The curve $F$ has complex multiplication, the order of the canonical bundle is $3$, so $\widetilde{S}\cong E\times F/(\IZ/3\IZ)$.
\item[(3)] The group $G$ is $\IZ/2\IZ \times \IZ/2\IZ$. Then $\widetilde{S}=E\times F/(\IZ/2\IZ)$.
\item[(4)] The group $G$ is $\IZ/4\IZ \times \IZ/2\IZ$. The curve $F$ has complex multiplication and $\widetilde{S}=E\times F/(\IZ/2\IZ)$.
\end{itemize}

Note that in the cases (2) and (4) the Picard rank of $\widetilde{S}$ is either 2 or 4, depending on whether $E$ is isogenous to $F$ or not. Also recall that if an abelian surface has Picard rank 4 and is isogenous to a product of elliptic curves, then it is isomorphic to a product, see \cite{SM}.

\begin{prop}
Let $A$ be an abelian surface which is isomorphic to a product of elliptic curves. Then any Fourier--Mukai partner of $A$ is isomorphic to $A$. In particular, the canonical cover of a bielliptic surface does not have any FM-partners in case (1). This also holds in cases (2), (3) and (4) if $\rho(\widetilde{S})=4$.
\end{prop}

\begin{proof}
Let $B$ be a FM-partner of $A$. By Orlov's results in \cite{Orlov1} there exists a Hodge isometry $T(A)\cong T(B)$ and, since by assumption the Neron--Severi group of $A$ contains a hyperbolic plane, we can again use \cite[Thm.\ 14.4.4]{Nikulin} to conclude that this isometry extends to a Hodge isometry $H^2(A,\IZ) \cong H^2(B,\IZ)$. By \cite[Thm.\ 1]{Shioda} this shows that $A\cong B$ (since $A$ is self-dual being a product of elliptic curves). 
\end{proof}

\begin{remark}
Two abelian varieties $A$ and $B$ are derived equivalent if and only if there exists a symplectic isomorphism $A\times \widehat{A}\cong B\times\widehat{B}$, see \cite{Orlov}. Thus, if in our situation $A=E\times F$ and $B=E'\times F'$ we have $E\times F\times E\times F\cong E'\times F'\times E'\times F'$. This is not sufficient to conclude that $A\cong B$: In \cite{Shioda2} Shioda gives a counterexample to such a statement even in smaller dimensions. Namely, there exist elliptic curves $C$, $C'$ and $C''$ such that $C\times C'\cong C\times C''$ but nevertheless $C' \ncong C''$. 
\end{remark}

The case (1) being dealt with, we now turn to case (2). We only have to consider the case where $E$ is not isogenous to $F$, so $\rho(\widetilde{S})=2$. The Neron--Severi group of $E\times F$ is generated by $E\times \left\{0\right\}$ and $\left\{0\right\} \times F$ with the two generators spanning the hyperbolic plane $U$. It is easy to see that $NS(E\times F/(\IZ/3\IZ))=U(3)$. We need the following

\begin{lem}
Let $L$ be the lattice $U(3)$. Then the canonical morphism $O(L)\rightarrow O(A_L)$ is surjective.
\end{lem}

\begin{proof}
It is easy to see that $O(L)$ is isomorphic to $\IZ/2\IZ\oplus \IZ/2\IZ$, the isometries being the identity $\id$, $-\id$, the map $\iota$ which interchanges the two generators of the hyperbolic plane $e$ and $f$ and the composition $-\id\circ \iota$. The group $A_L$ is $\IZ/3\IZ\oplus \IZ/3\IZ$. The bilinear form $b_L$ can be described as follows: The elements of the form $(x,0)$ resp.\ $(0,x)$ are isotropic, $b_L(x,x)=\frac{1}{3}$ for all $x$ and $b_L(x,y)=\frac{2}{3}$ for $x\neq y$, $x\neq 0, y\neq 0$. The surjectivity of the canonical map follows by an easy computation.
\end{proof}

\begin{lem}
The lattice $U(3)$ is the only one in its genus.
\end{lem}

\begin{proof}
A $p$-elementary lattice is determined by its discriminant, see for example \cite[Thm.\ 1.1]{AST}. Alternatively, one can use the classification of indefinite two-dimensional lattices as found for example in \cite[Ch.\ 15]{CS} and check that the discriminant forms of the other three two-dimensional lattices of determinant $-9$ are not isometric to the discriminant form of $U(3)$.
\end{proof}

\begin{cor}
Let $S$ be a bielliptic surface such that $\widetilde{S}=E\times F/(\IZ/3\IZ)$, where $E$ and $F$ are not isogenous. Then $FM(\widetilde{S})$ consists of the abelian surface $\widetilde{S}$ and its dual.
\end{cor}

\begin{proof}
This follows immediately from the two lemmas above and the counting formula recalled in Section \ref{Preliminaries}.
\end{proof}

We can now deal with the cases (3) and (4).

\begin{prop}
Let $A$ be an abelian surface such that $NS(A)$ is isometric to $U(2)$. Then $A$ has precisely one non-isomorphic Fourier--Mukai partner, namely $\widehat{A}$.
\end{prop}

\begin{proof}
The lattice $U(2)$ is even, indefinite and 2-elementary. Therefore, by Proposition \ref{Nik} any Hodge isometry of its orthogonal complement $T(A)$ extends to $H^2(A,\IZ)$. Using \cite[Thm.\ 1]{Shioda} this implies that any FM-partner of $A$ is either $A$ or $\widehat{A}$.
\end{proof}

The proposition immediately implies the following

\begin{cor}
Let $S$ be a bielliptic surface in case (3) or (4) such that $\rho(\widetilde{S})=2$. Then $\widetilde{S}$ has only one non-trivial FM-partner.\qqed
\end{cor}

Lastly, we have to consider case (3) under the assumption that $\rho(\widetilde{S})=3$. First note, that $NS(E\times E)=U\oplus \left\langle -2\right\rangle$ for a curve $E$ without complex multiplication, but we can also have $E\times F$ with $F$ isogenous to $E$ and in this case $NS(E\times F)=U\oplus \left\langle -2N\right\rangle$ for some $N\geq 1$. Dividing out $\IZ/2\IZ$ gives a lattice of the form $L=U(2)\oplus \left\langle -4N\right\rangle$. Arguing as in the last part of the proof of Proposition \ref{K3-Enriques}, we conclude that in this case there is also at most one non-trivial FM-partner. Hence, we proved the following  


\begin{prop}\label{abelian-bielliptic}
Let $S$ be a bielliptic surface.Then the canonical cover $A$ of $S$ has at most one non-trivial FM-partner, namely $\widehat{A}$.\qqed
\end{prop}

The above result has the following implication.

\begin{cor}\label{bir-kum}
If $A$ is an abelian surface as in the proposition and the generalised Kummer variety $K_n(A)$ ($n\geq 2$) is birational to $K_n(B)$ for some abelian surface $B$, then $B$ is either isomorphic to $A$ or $\widehat{A}$.
\end{cor}

\begin{proof}
This follows at once, since the assumption implies that $\Db(A)\cong \Db(B)$.
\end{proof}

\begin{remark}
Similar to the K3-case we expect the situation to become more complicated when we consider twisted surfaces.
\end{remark}


\begin{remark}
An interesting question is whether any FM-partner of an abelian surface admitting a principal polarisation also admits a principal polarisation. This would follow from a stronger statement, namely that a principally polarised abelian surface does not have any non-trivial FM-partners: This is not entirely implausible by the above results for surfaces which are products of elliptic curves and by results in \cite{Orlov} which state that the statement holds in the generic case, that is, when $\End(A)=\IZ$. Namely, for a (not necessarily principally polarised) generic abelian variety $A$ the number of FM-partners of $A$ is equal to $2^{s(A)}$, where $s(A)$ is the number of prime divisors of $\det(NS(A))/2$. 
\end{remark}

\section{Appendix: Exceptional and spherical objects on a bielliptic surface}

Let us now consider the structure of the derived category of a bielliptic surface. First a result which is implicit in \cite{BM2}.

\begin{prop}
Let $\Phi\colon \Db(S)\rightarrow \Db(S)$ be an autoequivalence of the derived category of a bielliptic surface $S$. Denote the kernel of $\Phi$ by $K$. Then $K$ is, up to shift, isomorphic to a sheaf.
\end{prop}

\begin{proof}
Recall that the statement holds for abelian varieties. By \cite[Thm.\ 4.5]{BM2} we know that there exists a lift of $\Phi$ to an autoequivalence of the canonical cover, which is an abelian surface. The kernel of the lift is given by $\widetilde{K}=(\pi\times \pi)^*(K)$, where $\pi\colon \widetilde{S} \rightarrow S$ is the canonical projection. Since $\pi$ is flat, the pullback does not need to be derived, and since $\widetilde{K}$ is up to shift a sheaf, the same holds for $K$.
\end{proof}

We now turn to the existence of special objects in the derived category. Recall that an object $E$ in the derived category $\Db(X)$ of a smooth projective variety $X$ of dimension $d$ is \emph{spherical} if $E\otimes \omega_X\cong E$ and $\Hom^i(E,E)=\IC$ for $i=0,d$ and $0$ otherwise. Here we use the notation $\Hom^i(E,E)$ for the group $\Hom(E,E[i])$; we will write $\hom^i(E,E)$ for the dimension of this space. Also recall that an object $E$ is called \emph{exceptional} if $\Hom(E,E)=\IC$ and $\Hom^i(E,E)=0$ for all $i\neq 0$. We start with the following easy

\begin{prop}
Let $X$ be a smooth projective variety of dimension $d$ whose canonical bundle is torsion of order 2 and let $\pi\colon \widetilde{X}\rightarrow X$ be the canonical cover. Then an object $E\in \Db(X)$ is exceptional if and only if $\pi^*E \in \Db(\widetilde{X})$ is spherical. 
\end{prop}

\begin{proof}
Let $E$ be an exceptional object. Then
\begin{align*}
\Hom^i(\pi^*E,\pi^*E)&=\Hom^i(E,\pi_*\pi^*E)=\Hom^i(E,E\otimes \pi_*\ko_{\widetilde{X}})=\\
&=\Hom^i(E,E)\oplus \Hom^i(E,E\otimes \omega_X).
\end{align*}
By Serre duality the dimension of the last summand is equal to the dimension of $\Hom^{d-i}(E,E)$ which shows that $\pi^*E$ is spherical.

Conversely, if $\pi^*E$ is spherical, then the above equation shows that either $E$ is exceptional or $\Hom^i(E,E)=0$ for all $i\neq d$ and $\IC$ for $i=d$. But if $\Hom^d(E,E)=\IC$, then by Serre duality $\Hom^0(E,E\otimes \omega_X)=\IC$. On the other hand $\Hom^0(E,E\otimes\omega_X)=\Hom^0(E\otimes\omega_X,E)$ and hence $\Hom^0(E,E)\neq 0$, a contradiction. Therefore, $E$ is an exceptional object.
\end{proof}

\begin{remark}
We can give a different proof of the ``only if'' direction as follows. We have already used that $\pi_*\ko_{\widetilde{X}}\cong \ko_X\oplus \omega_X$ (see \cite[Ch.\ I, Lem.\ 17.2]{BPV}). Furthermore, $R^i\pi_*\ko_{\widetilde{X}}=0$ for $i>0$ by \cite[Ch.\ III, Ex.\ 8.2]{Hartshorne} ($\pi$ is affine since it is finite). Hence there exists a triangle
\[\begin{xy}\xymatrix{\ko_X \ar[r] & R\pi_*\ko_{\widetilde{X}} \ar[r] & \omega_X }\end{xy}\]
in $\Db(X)$. This implies that $\pi$ is a simple morphism.

By \cite[Prop.\ 3.13]{ST} the pullback of an exceptional object $E \in \Db(X)$ is a spherical object in $\Db(\widetilde{X})$.
\end{remark}

\begin{cor}
Let $S$ be a bielliptic surface such that its canonical cover is of degree 2. Then there are no exceptional objects in $\Db(S)$.\qqed
\end{cor}

In fact, a more general statement holds.

\begin{prop}
If $S$ is a bielliptic surface, then no object $E$ in $\Db(S)$ is rigid, that is, $\Hom^1(E,E)\neq 0$ for all $E$. In particular, there are no spherical and exceptional objects.
\end{prop}

\begin{proof}
The main input is \cite[Lem.\ 15.1]{Bri}, which shows that there are no objects $F$ in the derived category of an abelian surface satisfying $\Hom^1(F,F)=0$. Recall that the canonical cover of $S$ is an abelian surface $A$ and $S$ is equal to the quotient of $A$ by a cyclic group $H$ whose generator will be denoted by $h$ and whose order will be denoted by $n$. The morphism $A\rightarrow S$ will be denoted by $\pi$.

Using \cite[Lem.\ 12.4]{Bri}, it is enough to show that there are no rigid sheaves. First, consider a torsion sheaf $T$. Then, since $T\otimes \omega_S\cong T$, $T=\pi_*F$ for some object $F \in  \Db(A)$ by \cite[Prop.\ 2.5]{BM2}. We have
\[\Hom^i(T,T)=\Hom^i(\pi_*F, \pi_*F)=\Hom^i(\pi^*\pi_*F,F)=\bigoplus_{k=1}^{n}\Hom^i((h^k)^*F,F).\]
The above shows that $F$ is rigid, which is impossible. Hence, no rigid torsion sheaves exist. Furthermore, no locally free sheaf $E$ can be rigid, because $\ko_S$ is a direct summand of $E\otimes E^\vee$ and hence $\IC=H^1(S,\ko_S)\subset \Ext^1(E,E)=\Hom^1(E,E)$. 

Now, let $E'$ be any rigid sheaf. By \cite[Lem.\ 2.2]{KO} its torsion subsheaf and the torsion-free quotient are rigid sheaves and by \cite[Cor.\ 2.3]{KO} a torsion-free rigid sheaf on a smooth projective surface is locally free. But we have seen that neither torsion nor locally free rigid sheaves can exist.   
\end{proof}

\begin{remark}
One can prove that there are no spherical objects using only the fact that there are no spherical objects in the derived category of an abelian surface. Let us demonstrate the technique in the case $n=2$. Thus,
\[\Hom^0(F,F)\oplus \Hom^0(h^*F,F)=\IC=\Hom^2(h^*F,F)\oplus \Hom^2(F,F).\]
If $\hom^0(F,F)=1$, then also $\hom^2(F,F)=1$ by Serre duality, hence $F$ is spherical which is absurd. If $\Hom^0(h^*F,F)=\IC$, then also $\Hom^0(F,h^*F)=\IC$, but $\Hom(h^*F,h^*F)=0$, a contradiction. For $n=3$ one can check that $h^*F\oplus F$ would have to be a spherical object and similar reasoning proves the other cases as well.
\end{remark}


\begin{remark}
Note that the situation is completely different for Enriques surfaces. Indeed, any line bundle $\kl$ on an Enriques surface $X$ is an exceptional object since $\Hom^i(\kl,\kl)=\Ext^i(\kl,\kl)\cong \Ext^i(\ko_X,\ko_X)\cong H^i(X,\ko_X)$. Also note that there are no spherical objects on a generic Enriques surface (see \cite[Prop.\ 3.17]{MMS}), but they do exist on non-generic ones.
\end{remark}


\end{document}